\documentclass[12 pt]{article}%
\usepackage{amsmath, amsfonts, amsthm, color,latexsym}
\usepackage{amsmath}
\usepackage{amsfonts}
\usepackage{amssymb}
\usepackage{color, soul}
\usepackage[all]{xy}
\usepackage{graphicx}%
\providecommand{\U}[1]{\protect\rule{.1in}{.1in}}
\allowdisplaybreaks[4]
\newtheorem{theorem}{Theorem}[section]
\newtheorem{proposition}[theorem]{Proposition}
\newtheorem{corollary}[theorem]{Corollary}
\newtheorem{example}[theorem]{Example}

\newtheorem{remark}[theorem]{Remark}

\newtheorem{lemma}[theorem]{Lemma}
\newtheorem{final remark}[theorem]{Final Remark}
\newtheorem{definition}[theorem]{Definition}
\allowdisplaybreaks[4]

\begin{document}

\title{Operators whose adjoints and second adjoints are almost Dunford-Pettis}
\author{Geraldo Botelho\thanks{Supported by Fapemig grants RED-00133-21 and APQ-01853-23.}\,\, and  Luis Alberto Garcia\thanks{Supported by a CAPES scholarship.\newline 2020 Mathematics Subject Classification: 46B42, 47B07, 47B65.\newline Keywords: Banach lattices, almost Dunford-Pettis operators, almost limited operators, order weakly compact operators.
}}
\date{}
\maketitle

\begin{abstract} First we characterize the Banach lattices $E$ whose biduals have the positive Schur property by means of second adjoints of operators on $E$ being almost Dunford-Pettis. 
Next we extend some known results concerning conditions on the Banach lattices $E$ and $F$ under which the adjoint $T^*$ and the second adjoint $T^{**}$ of any positive almost Dunford-Pettis operator $T \colon E \longrightarrow F$ are almost Dunford-Pettis. 
Finally, we prove when $T^*$ and $T^{**}$ are almost Dunford-Pettis for any (non necessarily almost Dunford-Pettis) $T$ that is either bounded, regular,  order bounded or weakly compact.

\end{abstract}

\section{Introduction} A linear operator between Banach spaces is a {\it Dunford-Pettis operator} if it sends weakly null sequences to norm null sequences. The term {\it completely continuous} is also often used.  The lattice counterpart of this important operator ideal is the class of almost Dunford-Pettis operators: an operator from a Banach lattice to a Banach space is {\it almost Dunford-Pettis} if it sends disjoint weakly null sequences to norm null sequences; or, equivalently, if it sends positive disjoint weakly null sequences to norm null sequences. Almost Dunford-Pettis operators have attracted the attention of many experts, for recent advances see \cite{ardakanichen, khazhak, luismonat, elfahri, pablov, oughajji, retbi, xiang}.

In this paper we address the duality problem for almost Dunford-Pettis operators, which has already been investigated before, see, e.g., \cite{wic, elbour, kamal}. Given a linear operator $T \colon E \longrightarrow F$ between Banach lattices,  our main focus is to establish conditions on $T$ and/or $E$ and/or $F$ so that the adjoint $T^{*} \colon F^* \longrightarrow E^{*}$ and/or the second adjoint $T^{**} \colon E^{**} \longrightarrow F^{**}$ are almost Dunford-Pettis. Some of the results we prove extend known results and some consider conditions which had  not been used before in this context.

To start describing the contents of the paper, recall that a Banach lattice $E$ has the {\it positive Schur property} if positive (or disjoint, or positive disjoint) weakly null sequences in $E$ are norm null. The literature on this property is extensive,  see \cite{joselucasvinicius} for very recent developments. In spite of the numerous known characterizations of the positive Schur property, we are not aware of any characterization of this property on the bidual $E^{**}$ of the Banach lattice $E$ which relies only on properties of $E$. In Section 2 we prove such a characterization by means of $T^{**}$ being almost Dunford-Pettis for every operator $T \colon E \longrightarrow \ell_\infty$.   

In Section  \ref{s3} we  address the problem of characterizing the Banach lattices $E$ and $F$ such that the adjoint $T^*$ of every positive almost Dunford-Pettis operator $T \colon E \longrightarrow F$ is almost Dunford-Pettis. The results we prove extend results from \cite{wic, kamal} and provide new situations where $T^*$ enjoys a property stronger than being almost Dunford-Pettis and situations where $T^{**}$ is almost Dunford-Pettis.

In Section 4 we investigate when $T^{\ast}$ or $T^{**}$ is almost Dunford-Pettis even if $T$ is not. The results we prove show when $T^*$ is almost Dunford-Pettis for $T$ order bounded or regular; and when $T^{**}$ is almost Dunford-Pettis for $T$ bounded order weakly compact or weakly compact.

For basic notions, notation and results on Banach lattices we refer to \cite{alibur, positiveoperators, nieberg}. Operators are always supposed to be linear. The symbol $x = o - \lim\limits_{n \to \infty} x_n$ means that $x$ is the order limit of the order convergent sequence $(x_n)_{n=1}^\infty$. For a Banach space $E$, by $J_E \colon E \longrightarrow E^{**}$ we denote the canonical embedding, which is also a Riesz homomorphism when $E$ is a Banach lattice.

\section{When $E^{**}$ has the positive Schur property} \label{s2}

A Banach space in which every weakly null sequence is norm null is said to have the {\it Schur property}. It follows from \cite[Corollary 11]{mujica2003} that second duals -- hence higher order duals -- of infinite dimensional Banach spaces never have the Schur property. For the positive Schur property the situation is quite different: Since AL-spaces have the positive Schur property and second duals of AL-spaces are AL-spaces, there are plenty of infinite dimensional Banach lattices $E$ for which $E^{**}$ has the positive Schur property. As far as we know, no characterization of such Banach lattices has appeared in the literature. Next we characterize such Banach lattices by means of second adjoints of operators on $E$ being almost Dunford-Pettis. This result shall be useful soon.


\begin{theorem}\label{fthe} 
The following are equivalent for a Banach lattice $E$. \\
{\rm (a)} For every bounded operator $T \colon E \longrightarrow \ell_\infty$, $T^{**}$ is almost Dunford-Pettis. \\
{\rm (b)} For every regular operator $T \colon E \longrightarrow \ell_\infty$, $T^{**}$ is almost Dunford-Pettis.\\
{\rm (c)} For every positive operator $T \colon E \longrightarrow \ell_\infty$, $T^{**}$ is almost Dunford-Pettis.\\
{\rm (d)} $E^{**}$ has the positive Schur property.
\end{theorem}

\begin{proof} It is enough to prove (c)  $\Rightarrow$ (d) because (d) $\Rightarrow$ (a) is easy and (a)  $\Rightarrow$ (b)  $\Rightarrow$ (c) are obvious.

The assumption is that $T^{**}$ is almost Dunford-Pettis for every positive $T \colon E \longrightarrow \ell_\infty$. Assume that $E^{\ast\ast}$ fails the positive Schur property. In this case there exists a positive disjoint weakly null non-norm null sequence $(z_{n}^{\ast\ast})_{n=1}^{\infty}$ em $E^{\ast\ast}$. By \cite[Proposition 2.1]{wic} we can suppose that this sequence is normalized. So, 
$$1=\|z_{n}^{\ast\ast}\|=\sup\{z_{n}^{\ast\ast}(z^{\ast}):0\leq z^{\ast}\in E^{\ast} \text{ and } \|z^{\ast}\|=1\} $$
for every $n\in\mathbb{N}$ (see \cite[p.\,182]{positiveoperators}). Therefore, there are $\alpha>0$ and a positive sequence $(z_{n}^{\ast})_{n=1}^{\infty}$ in $B_{E^{\ast}}$ such that $z_{n}^{\ast\ast}(z_{n}^{\ast})\geq \alpha$ for every $n\in \mathbb{N}$. Since the norm of $\ell_\infty$ is not order continuous, 
 \cite[Theorem 2.4.2]{nieberg} provides a vector $y\in \ell_\infty$  and a disjoint sequence  $(y_{n})_{n=1}^{\infty}$ in $\ell_\infty$ such that $0\leq y_{n}\leq y$ and $\|y_{n}\|=1$ for every $n\in\mathbb{N}$. Given a positive $z \in E$, for every natural number $m \in \mathbb{N}$,
$$0\leq s_{m}: =\sum_{n=1}^{m}z_{n}^{\ast}(z)y_{n}\leq \sum_{n=1}^{m}\|z_{n}^{\ast}\|\cdot \|z\|y_{n}=\|z\|\cdot \sum_{n=1}^{m}y_{n}=\|z\|\cdot \bigvee_{n=1}^{m}y_{n}\leq \|z\|y. $$
It follows that $0\leq s_{m}\uparrow\leq \|z\|y$. Since $\ell_\infty$ is Dedekind $\sigma$-complete, there exists $w\in \ell_\infty$ such that $s_{m}\uparrow w$. Now \cite[Theorem 1.14(iv)]{alibur} gives $w=o-\lim\limits_{m\rightarrow\infty}s_{m}$, proving that the sequence $(s_m)_{m=1}^\infty$ is order convergent, that is, the map
 $$T_0\colon E^{+}\longrightarrow \ell_\infty^{+}~,~ T_0(z)=o-\displaystyle\lim_{m\rightarrow\infty}\displaystyle\sum_{n=1}^{m}z_{n}^{\ast}(z)y_{n},$$ is well defined. From \cite[Theorem 1.14(ii)]{alibur} we know that $T_0$ is additive, hence $T_0$ admits a positive linear extension $T\colon E\longrightarrow \ell_\infty$ by Kantorovich's Theorem.  For all $n\in\mathbb{N}$, positive $z\in E$ and positive functional $y^{\ast}\in \ell_\infty^{\ast}$,  $T(z)=T_0(z)\geq z_{n}^{\ast}(z)y_{n}$ and $$T^{\ast}(y^{\ast})(z)=y^{\ast}(T(z))\geq y^{\ast}(z_{n}^{\ast}(z)y_{n})=y^{\ast}(y_{n})z_{n}^{\ast}(z),$$
which shows that $T^{\ast}(y^{\ast})\geq y^{\ast}(y_{n})z_{n}^{\ast}$. Moreover, for every positive $x^{\ast\ast}\in E^{\ast\ast}$,
$$T^{\ast\ast}(x^{\ast\ast})(y^{\ast})=x^{\ast\ast}(T^{\ast}(y^{\ast}))\geq x^{\ast\ast}(y^{\ast}(y_{n})z_{n}^{\ast})=x^{\ast\ast}(z_{n}^{\ast})y^{\ast}(y_{n})=x^{\ast\ast}(z_{n}^{\ast})J_{F}(y_{n})(y^{\ast}).$$
In particular,
$$ T^{\ast\ast}(z_{m}^{\ast\ast})\geq z_{m}^{\ast\ast}(z_{m}^{\ast})J_{F}(y_{m})\geq \alpha J_{F}(y_{m})\geq 0,$$
which implies
$$\|T^{\ast\ast}(z_{m}^{\ast\ast})\|\geq \alpha\|J_{F}(y_{m})\|=\alpha\|y_{m}\|=\alpha>0$$
for every $m\in\mathbb{N}$. Recalling that $(z_{n}^{\ast\ast})_{n=1}^{\infty}$ is weakly null we conclude that $T^{**}$ is not almost Dunford-Pettis. The proof is complete.
\end{proof}

\begin{remark}\label{rema}\rm  Using \cite[Theorem 4.56]{positiveoperators}, the proof above makes clear that, in items (a), (b) and (c) of the theorem, $\ell_\infty$ can be replaced by any Dedekind $\sigma$-complete Banach lattice containing a copy of $\ell_\infty$, in particular by any dual Banach lattice containing a copy of $\ell_\infty$.
\end{remark}


\section{Almost Dunford-Pettis adjoint operators}\label{s3}
It is easy to check that
$T^*$ almost Dunford-Pettis does not imply $T$ almost Dunford-Pettis in general, and that 
$T^{**}$ almost Dunford-Pettis always implies $T$ almost Dunford-Pettis. 
In this section we shall focus on the following questions: Given a positive operator $T \colon E \longrightarrow F$ between Banach lattices, is it true that: \\
(Q1) If $T$ is almost Dunford-Pettis, then so is its adjoint $T^* \colon F^* \longrightarrow E^*$?\\
(Q2) If $T$ is almost Dunford-Pettis, then so is its second adjoint $T^{**} \colon E^{**} \longrightarrow F^{**}$?

 Let us see that both questions have negative answers in general.

\begin{example}\rm For question (Q1), ${\rm id}_{\ell_1}$ is positive and almost Dunford-Pettis 
 but ${\rm id}_{\ell_1}^* = {\rm id}_{\ell_\infty}$ is not almost Dunford-Pettis because $\ell_\infty$ fails the positive Schur property \cite[p.\,82]{wnuk}.

 As to question (Q2), 
putting $ E = \left({\textstyle\bigoplus\limits_{n \in \mathbb{N}}} \ell_\infty^n \right)_1$, the identity on $E$ is positive and almost Dunford-Pettis because $E$ has the positive Schur property \cite[p.\,17]{wnuk}, whereas ${\rm id}_{E^{**}}= {\rm id}_E^{**}$ is not almost Dunford-Pettis because $E^{**}$ fails the positive Schur property \cite[Example 2.8]{jg}. There are two ways to give an example that is not an identity operator: (i) It is clear that the canonical embedding $J_E \colon E \longrightarrow E^{**}$ is almost Dunford-Pettis and it is easy to check that $J_E^{**}$ is not. (ii) Since $E^*$ contains a lattice copy of $\ell_1$ \cite[Example 2.8]{jg}, the norm of $E^{**}$ is not order continuous \cite[Theorem 2.4.14]{nieberg}. In particular, $E^{**}$ is a dual Banach lattice containing a copy of $\ell_\infty$ \cite[Theorem 4.56]{positiveoperators} and failing the positive Schur property. Bearing Remark \ref{rema} in mind, 
by Theorem \ref{fthe} there is a positive operator $T \colon E \longrightarrow E^{**}$, which is almost Dunford-Pettis because $E$ has the positive Schur property, such  that $T^{**}$  fails to be almost Dunford-Pettis.
\end{example}

The first purpose of this section is to extend the following two known results 
about question (Q1). 

\begin{theorem}\label{(i)} {\rm (Aqzzouz, Elbour and  Wickstead \cite[Theorem 5.1]{wic})} The following are equivalent for the Banach lattices $E$ and $F$.\\
{\rm (a)} If $T \colon E \longrightarrow F$  is positive and almost Dunford-Pettis, then so is $T^{\ast}$. \\
{\rm (b)} 
$E^{\ast}$ has order continuous norm or 
$F^{\ast}$ has the positive Schur property.
\end{theorem}

The following notions are well studied.\\
$\bullet$ A Banach lattice $E$ has the {\it dual positive Schur property} if positive weak$^*$-null sequences in $E^*$ are norm null (see \cite{wic}). \\
$\bullet$ A Banach lattice $E$ has {\it property (d)} if $(|x_n|)_{n=1}^\infty$ is weak$^*$-null for every disjoint weak$^*$-null sequence $(x_n)_{n=1}^\infty$ in $E^*$ (see \cite{lilian}).\\
$\bullet$ An operator $T$ from a Banach space $X$ to a Banach lattice $F$ is {\it almost limited} if $T^*$ sends disjoint weak$^*$-null sequences in $F^*$ to norm null sequences in $X^*$ (see \cite{elbour}).


\begin{theorem} \label{(ii)} {\rm (El Fahri and H'michane \cite[Theorem 4.3]{kamal})} Let $E$ be a Banach lattice and let $F$ be a Banach lattice with property (d). The following are equivalent.\\
{\rm (a)} If $T \colon E \longrightarrow F$  is positive and almost Dunford-Pettis, then $T$ is almost limited, hence $T^{\ast}$ is almost Dunford-Pettis. \\
{\rm (b)} 
$E^{\ast}$ has order continuous norm or 
$F$ has the dual positive Schur property.
\end{theorem}

The result \cite[Theorem 2.7]{nabil} is also related to Question (Q1), but it uses hypothesis we shall not deal with, so our results are not comparable to this one.

It is also known that sometimes a property stronger than being almost Dunford-Pettis holds for the adjoint of a positive almost Dunford-Pettis operator. 
Recalling that an operator between Banach spaces is {\it limited} if its adjoint sends weak$^*$-null sequences to norm null sequences (see Diestel and Bourgain \cite{diestel}), we define:
\begin{definition} \rm An operator $T$ from a Banach space $X$ to a Banach lattice $F$ is {\it positively limited} if $T^*$ takes positive weak$^*$-null sequences in $F^*$ to norm null sequences in $X^*$.
\end{definition}
It is clear that if $T$ is positively limited, then  $T^*$ is almost Dunford-Pettis. The converse fails: ${\rm id}_{c_0}^* = {\rm id}_{\ell_1}$ is almost-Dunford-Pettis but ${\rm id}_{c_0}$ is not positively limited.

For operators from Banach lattices to Banach spaces, the class of Dunford-Pettis operators is obviously contained in the class of almost Dunford-Pettis operators. To see that it is strictly smaller, just consider the identity operator on any Banach lattice having the positive Schur property and failing the Schur property, for instance, $L_1[0,1]$.



Our first result about Question (Q1) reads as follows.

\begin{theorem}\label{novot1} Let $E$ and $F$ be Banach lattices.\\
{\rm (a)} If  the adjoint of every positive Dunford-Pettis operator 
$T\colon E\longrightarrow F$ is almost Dunford-Pettis, then $E^{\ast}$ has order continuous norm or $F^{\ast}$ has the positive Schur property.\\
{\rm (b)} If every positive Dunford-Pettis operator $T\colon E\longrightarrow F$ is positively limited, the $E^{\ast}$ has order continuous norm or $F$ has the dual positive Schur property. 
\end{theorem}

\begin{proof} {\rm (a)} Suppose that the  norm  of $E^{\ast}$ is not order continuous and that $F^*$ fails the positive Schur property. The order discontinuity of the norm of $E^{\ast}$ gives a disjoint positive sequence  $(x_{n}^{\ast})_{n=1}^{\infty}$ in $E^{\ast}$ and a functional $x^{\ast}\in E^{\ast}$ so that $\|x_{n}^{\ast}\|=1$ and $0\leq x_{n}^{\ast}\leq x^{\ast}$ for every $n\in\mathbb{N}$ (see \cite[Theorem 2.4.2]{nieberg}). For each $n\in\mathbb{N}$ we can take a positive $x_{n} \in E$ such that $\|x_{n}\|=1$ and $x_{n}^{\ast}(x_{n})\geq \frac{1}{2}$. For every $m \in \mathbb{N}$ and any $x \in E$,
\begin{equation}\label{novoeq1}
\sum_{n=1}^{m}|x_{n}^{\ast}(x)|\leq \sum_{n=1}^{m}x_{n}^{\ast}(|x|)=\Big(\bigvee_{n=1}^{m}x_{n}^{\ast}\Big)(|x|)\leq x^{\ast}(|x|)<\infty,
\end{equation}
showing that the sequence $(x_{n}^{\ast}(x))_{n=1}^{\infty}$ is absolutely summable. Therefore, the map
$$S\colon E\longrightarrow \ell_{1}~,~S(x)=(x_{n}^{\ast}(x))_{n=1}^{\infty},$$
is a well defined positive Dunford-Pettis operator because $\ell_1$ has the Schur property. Using now that $F^{\ast}$ fails the positive Schur property, there is a disjoint positive normalized weakly null sequence $(y_{n}^{\ast})_{n=1}^{\infty}$ in $F^{\ast}$, and in this case we can take a positive normalized sequence
$(y_{n})_{n=1}^{\infty}$ in $F$ such that
$y_{n}^{\ast}(y_{n})\geq \frac{1}{2}$ for every $n\in\mathbb{N}$. It is immediate that
$$R\colon \ell_{1}\longrightarrow F ~,~R((z_{n})_{n=1}^{\infty})=\displaystyle\sum_{n=1}^{\infty}z_{n}y_{n},$$
 is a well defined positive linear operator. Define $T: =R\circ S\colon E\longrightarrow F,~ T(x)=\sum\limits_{n=1}^{\infty}x_{n}^{\ast}(x)y_{n}$. As $S$ is Dunford-Pettis, so is $T$. 

 The proof shall be complete once we prove that $T^{\ast}$ is not almost Dunford-Pettis. To do so, note that  $T(x)\geq x_{n}^{\ast}(x)y_{n}$ for all  $n\in\mathbb{N}$ and positive $x\in E$. Hence, for every positive  $y^{\ast}\in F^{\ast}$,
 $$T^{\ast}(y^{\ast})(x)=y^{\ast}(T(x))\geq y^{\ast}(x_{n}^{\ast}(x)y_{n})=y^{\ast}(y_{n})x_{n}^{\ast}(x).$$ In particular,  $T^{\ast}(y_{k}^{\ast})(x_{k})\geq y_{k}^{\ast}(y_{k})x_{k}^{\ast}(x_{k})$, hence $$\|T^{\ast}(y_{k}^{\ast})\|\geq T^{\ast}(y_{k}^{\ast})(x_{k})\geq y_{k}^{\ast}(y_{k})x_{k}^{\ast}(x_{k})\geq \frac{1}{4}$$
 for every $k\in\mathbb{N}$. This shows that $T^*$ is not almost Dunford-Pettis.

To prove (b) it is enough to replace, in the reasoning above, the sequence $(y_{n}^{\ast})_{n=1}^{\infty}$ in $F^{\ast}$ with a positive normalized weak$^{\ast}$-null sequence. \end{proof}

To prove our second theorem about Question (Q1) we need the following consequence of  \cite[Corollary 2.7]{dofre}.

\begin{lemma}\label{le4} A positive operator $T \colon E \longrightarrow F$ between Banach lattices is positively limited if and only if
$T^{\ast}(y_{n}^{\ast})(x_{n})\longrightarrow 0$ for every positive weak$^*$-null sequence $(y_{n}^{\ast})_{n=1}^{\infty}$ in $F^{\ast}$ and every positive disjoint bounded sequence $(x_{n})_{n=1}^{\infty}$ in $E$.
\end{lemma}

\begin{theorem}\label{novot2} Let $E, F$ be Banach lattices and let $X$ be a Banach space.\\
{\rm (a)} If  $E^{\ast}$ has order continuous norm, then every positive almost Dunford-Pettis operator $T\colon E\longrightarrow F$ is positively limited, hence  $T^{\ast}$ is almost Dunford-Pettis.\\
{\rm (b)} If  $F$ has the dual positive Schur property, then every continuous operator $T\colon X\longrightarrow F$ is positively limited, hence $T^{\ast}$ is almost Dunford-Pettis.
\end{theorem}

\begin{proof}   Let $(y_{n}^{\ast})_{n=1}^{\infty}$ be a positive weak$^*$-null sequence in  $F^{\ast}$.\\
 {\rm (a)} Let $T$ be a positive almost Dunford-Pettis operator and let $(x_{n})_{n=1}^{\infty}$ be a positive disjoint bounded sequence in $E$. Since $E^{\ast}$ has order continuous norm, by \cite[Theorem 2.4.14]{nieberg} we know that $(x_{n})_{n=1}^{\infty}$ is weakly null. So,  $(T(x_{n}))_{n=1}^\infty$ is norm null because $T$ is almost Dunford-Pettis. From
$$0\leq T^{\ast}(y_{n}^{\ast})(x_{n})=y_{n}^{\ast}(T(x_{n}))\leq \|y_{n}^{\ast}\| \cdot \|T(x_{n})\|\leq \displaystyle\sup_{k} \|y_{k}^{\ast}\|\cdot \|T(x_{n})\|\longrightarrow 0,$$
it follows that $T^{\ast}(y_{n}^{\ast})(x_{n})\longrightarrow 0$. Lemma \ref{le4} gives that $T$ is positively limited.\\
{\rm (b)} Supposing that $F$ has the dual positive Schur property and that $T$ is a continuous operator, we have $\| T^{\ast}(y_{n}^{\ast})\|\leq \|T^{\ast}\|\cdot \|y_{n}^{\ast}\|\longrightarrow 0$, which implies that $T$ is positively limited.
\end{proof}

To compare our theorems with the known results, we need the following lemma, which will be helpful later again. 

\begin{lemma}\label{novolem1}
Let $T \colon E \longrightarrow F$ be a positive operator between Banach lattices.\\ 
{\rm (a)} If $T$ is almost limited, then $T$ is positively limited.\\
{\rm (b)} If $F$ has property (d) and $T$ is positively limited, then $T$ is almost limited.
\end{lemma}
\begin{proof}
{\rm (a)} Suppose that $T$ is not positively limited. In this case, there is a positive weak$^*$-null sequence $(y_{n}^{\ast})_{n=1}^{\infty}$ in $F^{\ast}$ such that $(T^{\ast}(y_{n}^{\ast}))_{n=1}^\infty$ is not norm null.
We can take a subsequence $(y_{n_{j}}^{\ast})_{j=1}^{\infty}$ and $M > 0$ such that  $\|T^{\ast}(y_{n_{j}}^{\ast})\|\geq 2M>0$ for every $j\in\mathbb{N}$. Calling $w_{j}^{\ast}=y_{n_{j}}^{\ast}$ and $z_{j}^{\ast}=\frac{T^{\ast}(w_{j}^{\ast})}{\|T^{\ast}(w_{j}^{\ast})\|}$, each    $z_{j}^{\ast}$ is positive and, from
$$1=\|z_{j}^{\ast}\|=\sup\{z_{j}^{\ast}(z): z\in E^{+} \text{ and } \|z\|=1\}$$
\cite[p.\,182]{positiveoperators}, there exists a positive sequence  $(z_{j})_{j=1}^{\infty}$ in $E$ such that $\|z_{j}\|=1$ and $z_{j}^{\ast}(z_{j})\geq \frac{1}{2}$ for every $j\in\mathbb{N}$, that is, $w_{j}^{\ast}(T(z_{j}))\geq M$.
It is plain that $(w_{j}^{\ast})_{j=1}^\infty$ is weak$^*$-null. Setting $j_{1}=1$, we have $w_{j}^{\ast}(T(z_{j_{1}}))\longrightarrow 0$, so there exists  $j_{2}>j_{1}$ such that $w_{j_{2}}^{\ast}(T(z_{j_{1}}))<\frac{M}{2^{2\times 2+2}}$. The convergence $w_{j}^{\ast}(T(z_{j_{1}}+z_{j_{2}}))\longrightarrow 0$ provides $j_{3}>j_{2}$ such that $w_{j_{3}}^{\ast}(T(z_{j_{1}}+z_{j_{2}}))<\frac{M}{2^{2\times 3+2}}$. Inductively we construct a subsequence $(j_{m})_{m=1}^{\infty} \subseteq \mathbb{N}$ such that $w_{j_{m}}^{\ast}\left(\sum\limits_{k=1}^{m-1}T(z_{j_{k}})\right)\leq \frac{M}{2^{2m+2}}$ for every $m\geq 2$. Now we define
$$x_{m}=\left(T(z_{j_{m}})-4^{m}\cdot\displaystyle\sum_{k=1}^{m-1}T(z_{j_{k}})-2^{-m}\cdot\sum_{k=1}^{\infty}2^{-k}T(z_{j_{k}})\right)^{+} \text{ for each } m\geq 2.$$
Taking $x_{1}=0$, the sequence $(x_{m})_{m=1}^{\infty}$ is positive disjoint in $F$ by \cite[Lemma 2.6]{retbi} or, alternatively, by the proof of \cite[Lemma 2.4]{pagter}. From $0\leq x_{m}\leq T(z_{j_{m}})$ we get $\|x_{m}\|\leq \|T(z_{j_{m}})\|\leq \|T\|$ for every $m\in\mathbb{N}$. Taking $N>0$ so that $\|w_{j_{m}}^{\ast}\|\leq N$, for every $m\geq 2$ we have
 \begin{align}
w_{j_{m}}^{\ast}(x_{m})&=w_{j_{m}}^{\ast}\left(T(z_{j_{m}})-4^{m}\cdot\displaystyle\sum_{k=1}^{m-1}T(z_{j_{k}})-2^{-m}\cdot\sum_{k=1}^{\infty}2^{-k}T(z_{j_{k}})\right)^{+}\nonumber\\
&\geq w_{j_{m}}^{\ast}\left(T(z_{j_{m}})-4^{m}\cdot\displaystyle\sum_{k=1}^{m-1}T(z_{j_{k}})-2^{-m}\cdot\sum_{k=1}^{\infty}2^{-k}T(z_{j_{k}})\right)\nonumber\\
&=w_{j_{m}}^{\ast}(T(z_{j_{m}}))-4^{m}w_{j_{m}}^{\ast}\left(\displaystyle\sum_{k=1}^{m-1}T(z_{j_{k}})\right)-2^{-m}w_{j_{m}}^{\ast}\left(\sum_{k=1}^{\infty}2^{-k}T(z_{j_{k}})\right)\nonumber\\
&\geq M-\frac{M4^{m}}{2^{2m+2}}-\frac{N\|T\|}{2^{m}}=M(1-\frac{1}{4})-\frac{N\|T\|}{2^{m}}=\frac{3M}{4}-\frac{N\|T\|}{2^{m}}.\label{yn3z}
\end{align}
Since the sequence $(x_{m})_{m=1}^{\infty}$ is disjoint, by \cite[Exercise 22. p.\,77]{positiveoperators} there exists a disjoint positive sequence  $(u_{m}^{\ast})_{m=1}^{\infty}$ in $F^{\ast}$ such that $0\leq u_{m}^{\ast}\leq w_{j_{m}}^{\ast}$ and $u_{m}^{\ast}(x_{m})= w_{j_{m}}^{\ast}(x_{m})$ for every $m\in\mathbb{N}$. Since $(w_{j_{m}}^{\ast})_{n=1}^\infty$ is weak$^*$-null, so is  $(u_{m}^{\ast})_{j=1}^\infty$, and $\|T^{\ast}(u_{m}^{\ast})\|\longrightarrow 0$ because $T$ is almost limited. From (\ref{yn3z}) it follows that
$$\frac{3M}{4}-\frac{N\|T\|}{2^{m}}\leq w_{j_{m}}^{\ast}(x_{m})=u_{m}^{\ast}(x_{m})\leq u_{m}^{\ast}(T(z_{j_{m}}))=T^{\ast}(u_{m}^{\ast})(z_{j_{m}})\leq \|T^{\ast}(u_{m}^{\ast})\|$$
for every $m$. Making $m\rightarrow\infty$ we get $M\leq 0$, a contradiction which proves that $T$ is positively limited.

Statement {\rm (b)} follows from \cite[Theorem 3]{elbour}.
\end{proof}

Now we are ready to compare our Theorems \ref{novot1} and \ref{novot2} to Theorems  \ref{(i)} and  \ref{(ii)}.\\
$\bullet$ Since the class of Dunford-Pettis operators is strictly smaller than the class of almost Dunford-Pettis operators,  Theorem \ref{novot1}(a) obtains, with a weaker assumption, the same conclusion of Theorem \ref{(i)}(a)$\Rightarrow$(b). 
\\
$\bullet$ For the same reason, Theorem \ref{novot1}(b) obtains, with a weaker assumption, the same conclusion of Theorem \ref{(ii)}(a)$\Rightarrow$(b).\\
$\bullet$ Theorem  \ref{novot2} obtains the same conclusion of Theorem \ref{(ii)}(b)$\Rightarrow$(a) without asking $F$ to have propriedade (d). In this case, if $F$ has property (d), then $T$ is almost limited by Lemma \ref{novolem1}.






\medskip

Turning our attention to Question (Q2), to the best of our knowledge, all that is known is the iteration of the known results about Question (Q1), namely Theorems \ref{(i)} and  \ref{(ii)}. Theorem \ref{(i)} can be iterated with itself in four different ways, but the only interesting outcome is the following:

\begin{corollary}\label{1cor1} Let $E$ and $F$ be Banach lattices such that $E^*$ and $F^{**}$ have order continuous norms. Then $T^{**}$ is almost Dunford-Pettis whenever $T \colon E \longrightarrow F$ is positive and almost Dunford-Pettis.
\end{corollary}

Indeed, in two of the other iterations, $E^{**}$ has the positive Schur property, therefore the result is trivial in this case. In the remaining iteration, $F^*$ has  the positive Schur property and $F^{**}$ has order continuous norm. Let us see that there exists no such $F$: the positive Schur property of $F^*$ yields that $F^*$ contains a lattice copy of $\ell_1$, and in this case the norm of $F^{**}$ is not order continuous \cite[Theorem 2.4.14] {nieberg}.

Now we study the iteration of Theorem \ref{(ii)} with itself. Let an operator $T$ between Banach lattices be given. Since Dedekind $\sigma$-complete Banach lattices have property (d) and $T^*$  is defined on a Dedekind complete Banach lattice, from Lemma \ref{novolem1} we know that $T^*$ is almost limited if and only if it is positively limited.

Again, although Theorem \ref{(ii)} can be iterated with itself in four different ways, the only interesting outcome is the following:

\begin{corollary}\label{(iv)} Let $E$ and $F$ be Banach lattices such that $F$ has property (d) and $E^{\ast}$ and $F^{\ast\ast}$ have order continuous norms. If $T \colon E \longrightarrow F$ is positive and almost Dunford-Pettis, then $T^{\ast}$ is positively limited, that is, $T^*$ is almost limited, hence $T^{**}$ is almost Dunford-Pettis.
\end{corollary}

Indeed, in two of the other iterations, $E^*$ has the dual positive Schur property, therefore the result is trivial in this case. In the remaining iteration, $F$ has the dual positive Schur property and $F^{**}$ has order continuous norm. Again, there is no such $F$ because the dual positive Schur property of $F$ implies that $F^*$ has the positive Schur property and, as mentioned above, there is no such a Banach lattice.

%
%

The applications of Theorem \ref{(i)} first and then Theorem \ref{(ii)} and vice-versa shall be discussed in due time. Our result about Question (Q2) reads as follows.

\begin{theorem} \label{teo2} Let $E, F$ be Banach lattices such that $E^{\ast}$ and $F$ haver order continuous norms.  If $T \colon E \longrightarrow F$ is positive and almost Dunford-Pettis, then $T^{\ast}$ is positively limited, that is, $T^*$ is almost limited, hence $T^{**}$ is almost Dunford-Pettis.
\end{theorem}

\begin{proof} Let a positive disjoint bounded sequence $(y_{n}^{\ast})_{n=1}^{\infty}$ in $F^{\ast}$ be given. Since $F$ has order continuous norm, by \cite[Theorem 2.4.3]{nieberg} we know that $(y_{n}^{\ast})_{n=1}^{\infty}$ is weak$^*$-null.
From Theorem \ref{novot2} we know that $T$ is positively limited because $E^*$ has order continuous norm, so $(T^{\ast}(y_{n}^{\ast}))_{n=1}^\infty$ is norm null in $E^*$. For any positive weak$^*$-null, hence bounded, sequence $(x_{n}^{\ast\ast})_{n=1}^{\infty}$ in $E^{\ast\ast}$, we have
$$0\leq T^{\ast\ast}(x_{n}^{\ast\ast})(y_{n}^{\ast})=x_{n}^{\ast\ast}(T^{\ast}(y_{n}^{\ast}))\leq \|x_{n}^{\ast\ast}\| \cdot \|T^{\ast}(y_{n}^{\ast})\|\leq \sup_k \|x_k^{**}\| \cdot \|T^{\ast}(y_{n}^{\ast})\|\longrightarrow 0,$$
showing that $ T^{\ast\ast}(x_{n}^{\ast\ast})(y_{n}^{\ast})\longrightarrow 0.$ Calling on Lemma \ref{le4} we conclude that $T^{\ast}$ is positively limited, that is, $T^{\ast}$ is almost limited.
\end{proof}

Now we stress that our Theorem \ref{teo2} improves upon the known cases, namely Corollaries \ref{1cor1} and \ref{(iv)}.\\
(i) Theorem \ref{teo2} gets a conclusion stronger than the one in Corollary \ref{1cor1} with a weaker assumption.\\
(ii) Theorem \ref{teo2} gets the same conclusion of Corollary \ref{(iv)} with a weaker assumption and without asking $F$ to have property (d).

\begin{remark}\rm 
As to the other iterations of Theorems \ref{(i)} and \ref{(ii)}, we have:\\
$\bullet$ If one applies Theorem \ref{(i)} first and then Theorem \ref{(ii)}, then one ends up with a version of Corollary \ref{(iv)} without the assumption on property (d). We have just seen in (ii) that Theorem \ref{teo2} improves this result. \\
$\bullet$ Applying Theorem \ref{(ii)} first and then Theorem \ref{(i)}, we get a version of Theorem \ref{teo2} asking $F$ to have property (d), which is, of course, less general than Theorem \ref{teo2}.
\end{remark}

\section{When $T^*$ or $T^{**}$ is almost Dunford-Pettis even if $T$ is not}\label{s5}

The title of this section is self-explanatory. We begin by giving conditions, different from the ones of the previous sections and of the results in the literature we are aware of, under which adjoints of order bounded/regular operators are almost Dunford-Pettis.

\begin{proposition}
Let $F$ be a Banach lattice and let $K$ be a compact Hausdorff space. Every order bounded operator $T\colon C(K) \longrightarrow F$ is positively limited, hence  $T^{\ast}$ is almost Dunford-Pettis.
\end{proposition}
\begin{proof} Let $(y_{n}^{\ast})_{n=1}^{\infty}$ be a positive  weak$^*$-null sequence in $F^{\ast}$. For every $n$,
$$\|T^{\ast}(y_{n}^{\ast})\|=\sup\{|T^{\ast}(y_{n}^{\ast})(f)|:f\in B_{C(K)}\}\leq \sup\{y_{n}^{\ast}(|T(f)|):f\in B_{C(K)}\}.$$
For $f\in B_{C(K)}$ we have $|f|\leq \textbf{1}$ in $C(K)$, thus, since $T$ order bounded, there is $y\in F^{+}$ such that $|T(f)|\leq y$. Therefore, 
$\|T^{\ast}(y_{n}^{\ast})\|\leq y_{n}^{\ast}(y)\longrightarrow 0$, which shows that $T$ is positively limited.
\end{proof}

\begin{proposition}
Let $E, F$ be Banach lattices with $E^{\ast}$ having the the positive Schur property. The adjoint of every regular operator $T\colon E \longrightarrow F$ is almost Dunford-Pettis.
\end{proposition}
\begin{proof} Write $T=T_{1}-T_{2}$ where $T_{1}, T_{2}\colon  E \longrightarrow F$ are positive operators. Thus $T^{\ast}=T_{1}^{\ast}-T_{2}^{\ast}$. Given a  positive weakly null sequence $(y_{n}^{\ast})_{n=1}^{\infty}$ in $F^{\ast}$, the sequences $(T_{i}^{\ast}(y_{n}^{\ast}))_{n=1}^{\infty}, i=1,2$, are positive and weakly null in $E^{\ast}$, hence  $\|T_{i}^{\ast}(y_{n}^{\ast})\|\longrightarrow 0$. It follows that $\|T^{\ast}(y_{n}^{\ast})\|\leq \|T_{1}^{\ast}(y_{n}^{\ast})\|+\|T_{2}^{\ast}(y_{n}^{\ast})\|\longrightarrow 0$, proving that $T^{\ast}$ is almost Dunford-Pettis.
\end{proof}

\begin{proposition}\label{novopropo1} If the Banach lattice $E$ has the dual positive Schur property, then, no matter the Banach lattice $F$, every regular  operator $T \colon E \longrightarrow F$ is positively limited, hence $T^{\ast}$ is almost Dunford-Pettis.
\end{proposition}
\begin{proof} Let $T_1, T_2 \colon E \longrightarrow F$ be positive operators such that $T=T_{1}-T_{2}$. 
Given a positive weak$^*$-null sequence $(y_{n}^{\ast})_{n=1}^{\infty}$ in $F^{\ast}$, for all  $n\in\mathbb{N}$ and $i = 1,2$, $T_{i}^{\ast}(y_{n}^{\ast})$ is a positive functional. For every $x\in E$, $(T_1+T_2)(|x|) \in F$, so
\begin{align*}
|(|T^{\ast}(y_{n}^{\ast})|)(x)|&=|(|T_{1}^{\ast}(y_{n}^{\ast})-T_{2}^{\ast}(y_{n}^{\ast}))|(x)|\leq |T_{1}^{\ast}(y_{n}^{\ast})-T_{2}^{\ast}(y_{n}^{\ast})|(|x|)\\
&\leq (T_{1}^{\ast}(y_{n}^{\ast})+T_{2}^{\ast}(y_{n}^{\ast}))(|x|)=T_{1}^{\ast}(y_{n}^{\ast})(|x|)+T_{2}^{\ast}(y_{n}^{\ast})(|x|)\\
&=y_{n}^{\ast}(T_{1}(|x|))+ y_{n}^{\ast}(T_{2}(|x|))=y_{n}^{\ast}((T_{1} +T_{2})(|x|))\longrightarrow 0,
\end{align*}
proving that $(|T^{\ast}(y_{n}^{\ast})|)_{n=1}^\infty$ is positive and weak$^*$-null in $E^*$. By assumption we have $\|T^*(y_n^*)\| \longrightarrow 0$.  
\end{proof}

Recall that an operator $T$ from a Banach lattice $E$ to a Banach space $X$ is {\it order weakly compact} if $T([0,x])$ is relatively weakly compact in $X$ for every positive $x \in E$ (see \cite[Section 3.4, p.\,191]{nieberg} and \cite[p.\,318]{positiveoperators}). By \cite[Theorem 5.57]{positiveoperators} it follows easily that every almost Dunford-Pettis operator is order weakly compact. 
The identity operator on $c_0$ shows that the converse does not hold in general.

To motivate the next theorem, note that, on the one hand, the identity operator on $\ell_2$ is obviously weakly compact, hence order weakly compact. On the other hand, arguing with the sequence of canonical unit vectors we see easily that its adjoint, which coincides with itself, is not positively limited. 




\begin{theorem}\label{teo5} If  the Banach lattice $E$ has the dual positive Schur property, then, no matter the Banach space $X$, the adjoint $T^*$ of every bounded  order weakly compact operator $T \colon E \longrightarrow X$ is almost limited, hence $T^{**}$ is almost Dunford-Pettis.
\end{theorem}
\begin{proof}  Given an order weakly compact operator $T \colon E \longrightarrow X$, 
by \cite[Theorem 5.58]{positiveoperators} there are a Banach lattice $F$ with order continuous norm, a Riesz homomorphism  $R\colon E\longrightarrow F$ and a bounded linear operator $S\colon F\longrightarrow X$ such that $T=S\circ R$. By Proposition \ref{novopropo1} we have that $R$ is positively limited. 
Let $(x_{n}^{\ast\ast})_{n=1}^{\infty}$ be a positive weak$^*$-null sequence in $E^{\ast\ast}$ 
and let $(y_{n}^{\ast})_{n=1}^{\infty}$ be a positive disjoint bounded sequence in  $F^{\ast}$. Since $F$ has order continuous norm, by \cite[Theorem 2.4.3]{nieberg} we know that  $(y_{n}^{\ast})_{n=1}^{\infty}$ is weak$^*$-null, hence 
$(R^{\ast}(y_{n}^{\ast}))_{n=1}^\infty$ is norm null in $E^*$ because $R$ is positively limited. Therefore, 
$$0\leq R^{\ast\ast}(x_{n}^{\ast\ast})(y_{n}^{\ast})=x_{n}^{\ast\ast}(R^{\ast}(y_{n}^{\ast}))\leq \|x_{n}^{\ast\ast}\|\cdot \|R^{\ast}(y_{n}^{\ast})\|\leq \sup_{k} \|x_{k}^{\ast\ast}\| \cdot \|R^{\ast}(y_{n}^{\ast})\|\longrightarrow 0,$$
which proves that $ R^{\ast\ast}(x_{n}^{\ast\ast})(y_{n}^{\ast})\longrightarrow 0$. Calling on Lemma \ref{le4} once again we get $\| R^{\ast\ast}(x_{n}^{\ast\ast})\|\longrightarrow 0$, thus
$$\|T^{\ast\ast}(x_{n}^{\ast\ast})\|=\|(S\circ R)^{\ast\ast}(x_{n}^{\ast\ast})\|=\|S^{\ast\ast}( R^{\ast\ast}(x_{n}^{\ast\ast}))\|\leq \|S^{\ast\ast}\|\cdot \|R^{\ast\ast}(x_{n}^{\ast\ast})\|\longrightarrow 0.$$
It follows that $(T^{\ast\ast}(x_{n}^{\ast\ast}))_{n=1}^\infty $ is norm null. Using that dual Banach lattices are  Dedekind complete, hence have property (d),  we conclude that $T^{\ast}$ is almost limited.
\end{proof}

Since weakly compact operators, L-weakly compact operators, M-weakly compact operators and order L-weakly compact operators are all order weakly compact (see \cite[Sections 5.2 and 5.3, Theorem 5.61 and p.\,319]{positiveoperators} and \cite{moussa}), Theorem \ref{teo5} applies to all these classes of operators. 
Moreover, the following consequence holds.

\begin{corollary} For all Banach spaces $X$ and Hausdorff compact spaces $K$, the adjoint $T^*$ of every weakly compact operator $T \colon C(K) \longrightarrow X$ is almost limited, hence $T^{**}$ is almost Dunford-Pettis.
\end{corollary}

To motivate the next consequence of Proposition \ref{novopropo1}, let us see that adjoints of bounded operators are not always almost limited, even if the operator takes values in a Banach lattice with the dual positive Schur property. 

\begin{example} \rm In this example, $B_{c_{0}^{\ast}}$ is endowed with the topology induced by the weak$^*$ topology on $c_0^*$. Let $T\colon c_{0}\longrightarrow C[0,1]$ be defined by $T(x)(t)=f(t)(x)$, where $f\colon [0,1]\longrightarrow B_{c_{0}^{\ast}}$ is a surjective continuous function, whose existence is guaranteed by the Hahn-Mazurkiewicz Theorem. It is clear that $T$ is linear and  $\|T(x)\|=\|x\|$ for every $x\in c_{0}$. 
Let us see that the adjoint $T^{\ast}\colon C[0,1]^{\ast}\longrightarrow \ell_{1}$ 
fails to be almost limited. Indeed,  the sequence $(J_{c_{0}}(e_{n}))_{n=1}^{\infty}$ is disjoint, positive and weak$^*$-null in $\ell_{\infty}$; and, for each  $n\in\mathbb{N}$,
$$\|T^{\ast\ast}(J_{c_{0}}(e_{n}))\|=\|J_{C[0,1]}(T(e_{n}))\|=\|T(e_{n})\|=\|e_{n}\|=1.$$
\end{example}

Nevertheless, as adjoints of regular operators are regular, the next result follows from the application of Proposition \ref{novopropo1} to the adjoint $T^{\ast}\colon F^{\ast}\longrightarrow E^{\ast}$.  

\begin{corollary}
 Let $F$  be a Banach lattice such that 
 $F^{\ast}$ has the dual positive Schur property. Then, no matter the Banach lattice $E$, the adjoint $T^*$ of every regular operator $T \colon E \longrightarrow F$ is almost limited, hence $T^{**}$ is almost Dunford-Pettis.
\end{corollary}

Our final result concerns second adjoints of weakly compact operators on Banach lattices $E$ eventually not satisfying the assumption of Theorem \ref{teo5}.

\begin{proposition} Let $T\colon E\longrightarrow F$ be a weakly compact operator between Banach lattices and suppose that one of the following holds:\\ 
\indent {\rm (a)} $T$ is regular and $F$ has the positive Schur property.\\
\indent {\rm (b)} $T$ is positive, $F$ has order continuous norm and $E^{\ast}$ has the positive Schur property.\\
Then $T^{\ast}$ is almost limited, hence $T^{\ast\ast}$ is almost Dunford-Pettis.
\end{proposition}
\begin{proof}
 Let $(x_{n}^{\ast\ast})_{n=1}^{\infty}$ be a positive weak$^*$-null sequence in $E^{\ast\ast}$. The weak compactness of $T$ gives $T^{\ast\ast}(E^{\ast\ast})\subseteq J_{F}(F)$ \cite[Theorem 3.5.8]{meg}, so for each $n\in\mathbb{N}$ there exists $y_{n}\in F$ such that $T^{\ast\ast}(x_{n}^{\ast\ast})=J_{F}(y_{n})$. Supposing (a), there are positive operators $T_{1},T_{2}\colon E \longrightarrow F$ so that $T=T_{1}-T_{2}$, hence $T^{\ast\ast}=T_{1}^{\ast\ast}-T_{2}^{\ast\ast}$. Since $J_{F}$ is a Riesz  homomorphism and $(x_n^{**})_{n=1}^\infty$ is weak$^*$-null, for every $y^{\ast}\in F^{\ast}$,
\begin{align*}
|y^{\ast}(|y_{n}|)|&\leq |y^{\ast}|(|y_{n}|)=J_{F}(|y_{n}|)(|y^{\ast}|)=|J_{F}(y_{n})|(y^{\ast})\\
&=|T^{\ast\ast}(x_{n}^{\ast\ast})|(|y^{\ast}|)
\leq |T^{\ast\ast}|(x_{n}^{\ast\ast})(|y^{\ast}|)\\
&= |T_{1}^{\ast\ast}-T_{2}^{\ast\ast}|(x_{n}^{\ast\ast})(|y^{\ast}|)\leq (T_{1}^{\ast\ast}+T_{2}^{\ast\ast})(x_{n}^{\ast\ast})(|y^{\ast}|) \\
&=T_{1}^{\ast\ast}(x_{n}^{\ast\ast})(|y^{\ast}|)+T_{2}^{\ast\ast}(x_{n}^{\ast\ast})(|y^{\ast}|)=x_{n}^{\ast\ast}((T_{1}^{\ast}+T_{2}^{\ast})(|y^{\ast}|))\longrightarrow 0,
\end{align*}
showing that the sequence $(|y_{n}|)_{n=1}^\infty$ is weakly null. The positive Schur property of $F$ gives 
$$\|T^{\ast\ast}(x_{n}^{\ast\ast})\|=\|J_{F}(y_{n})\|=\|y_{n}\| =\| |y_{n}|\|\longrightarrow 0,$$
which proves that $T^{\ast}$ is almost limited.

Now assume (b) and let $(y_{n}^{\ast})_{n=1}^{\infty}$ be a positive disjoint bounded sequence in $F^{\ast}$. The order continuous norm of $F$ implies that the sequence  $(y_{n}^{\ast})_{n=1}^{\infty}$ is weak$^*$-null \cite[Theorem 2.4.3]{nieberg}. The weak compactness of $T$ implies that $T^{\ast}$ is weak$^*$-weak continuous \cite[Theorem 3.5.14]{meg}, so  $(T^{\ast}(y_{n}^{\ast}))_{n=1}^\infty$ is weakly null in $E^{\ast}$. It follows that $\|T^{\ast}(y_{n}^{\ast})\|\longrightarrow 0$ because  each $T^{\ast}(y_{n}^{\ast})$ is positive and $E^{\ast}$ has the positive Schur property. From
\begin{align*}
|y_{n}^{\ast}(y_{n})|&\leq y_{n}^{\ast}(|y_{n}|)=J_{F}(|y_{n}|)(y^{\ast}_{n})=|J_{F}(y_{n})|(y^{\ast}_{n})=|T^{\ast\ast}(x_{n}^{\ast\ast})|(y^{\ast}_{n})\\
&= T^{\ast\ast}(x_{n}^{\ast\ast})(y^{\ast}_{n})=x_{n}^{\ast\ast}(T^{\ast}(y^{\ast}_{n}))\leq \|x_{n}^{\ast\ast}\| \|T^{\ast}(y^{\ast}_{n})\|\longrightarrow 0,
\end{align*}
 we get $y_{n}^{\ast}(y_{n})\longrightarrow 0$. From \cite[Corollary 2.6]{dofre} it follows that $\|T^{\ast\ast}(x_{n}^{\ast\ast})\|=\|J_{F}(y_{n})\|=\|y_{n}\|\longrightarrow 0$ and the proof is complete.
\end{proof}

\bigskip
\noindent G. Botelho ~~~~~~~~~~~~~~~~~~~~~~~~~~~~~~~~~~~~~~~~\,L. A. Garcia \\Faculdade de Matem\'atica~~~~~~~~~~~~~~~~~~~~~~Instituto de Ciências Exatas\\
Universidade Federal de Uberl\^andia~~~~~~~~ Universidade Federal de Minas Gerais\\
38.400-902 -- Uberl\^andia -- Brazil~~~~~~~~~~~~ 31.270-901 -- Belo Horizonte -- Brazil\\
e-mail: botelho@ufu.br ~~~~~~~~~~~~~~~~~~~~~~~~~e-mail: garcia\_1s@hotmail.com

%
%
%

\end{document}